\documentclass[reqno,11pt]{amsart}
\usepackage{amsmath, latexsym, amsfonts, amssymb, amsthm, amscd}
\usepackage[english,frenchb]{babel} 
\usepackage[utf8]{inputenc}
\usepackage{graphics,epsf,psfrag,epsfig}

\newenvironment{poliabstract}[1]
  {\renewcommand{\abstractname}{#1}\begin{abstract}}
  {\end{abstract}}

\usepackage{setspace}

\usepackage{anysize}
\marginsize{1in}{1in}{1in}{1in}

\numberwithin{equation}{section}

\newtheorem{theorem}{Theorem}[section]
\newtheorem{lemma}[theorem]{Lemma}
\newtheorem{proposition}[theorem]{Proposition}
\newtheorem{cor}[theorem]{Corollary}
\newtheorem{rem}[theorem]{Remark}

\DeclareMathOperator{\argmin}{\mathrm{argmin}}
\renewcommand{\ge}{\geq}
\renewcommand{\le}{\leq}
\newcommand{\ind}{\mathbf{1}}

\newcommand{\R}{\mathbb{R}}
\newcommand{\Z}{\mathbb{Z}}
\newcommand{\N}{\mathbb{N}}
\renewcommand{\tilde}{\widetilde}

\newcommand{\cT}{{\ensuremath{\mathcal T}} }

\newcommand{\bP}{{\ensuremath{\mathbf P}} }
\newcommand{\bE}{{\ensuremath{\mathbf E}} }

\DeclareMathSymbol{\leqslant}{\mathalpha}{AMSa}{"36} 
\DeclareMathSymbol{\geqslant}{\mathalpha}{AMSa}{"3E} 
\DeclareMathSymbol{\eset}{\mathalpha}{AMSb}{"3F}     
\renewcommand{\leq}{\;\leqslant\;}                   
\renewcommand{\geq}{\;\geqslant\;}                   
\newcommand{\dd}{\,\text{\rm d}}             

\newcommand{\bbE}{{\ensuremath{\mathbb E}} }

\newcommand{\bbN}{{\ensuremath{\mathbb N}} }

\newcommand{\bbP}{{\ensuremath{\mathbb P}} }

\newcommand{\bbR}{{\ensuremath{\mathbb R}} }

\newcommand{\bbZ}{{\ensuremath{\mathbb Z}} }

\newcommand{\gga}{\gamma}            
\newcommand{\gd}{\delta}
\newcommand{\gep}{\varepsilon}       

\newcommand{\gD}{\Delta}

\newcommand{\go}{\omega}

\newcommand{\gl}{\lambda}

\makeatletter
\def\captionfont@{\footnotesize}
\def\captionheadfont@{\scshape}

\long\def\@makecaption#1#2{%
  \vspace{2mm}
  \setbox\@tempboxa\vbox{\color@setgroup
    \advance\hsize-6pc\noindent
    \captionfont@\captionheadfont@#1\@xp\@ifnotempty\@xp
        {\@cdr#2\@nil}{.\captionfont@\upshape\enspace#2}%
    \unskip\kern-6pc\par
    \global\setbox\@ne\lastbox\color@endgroup}%
  \ifhbox\@ne 
    \setbox\@ne\hbox{\unhbox\@ne\unskip\unskip\unpenalty\unkern}%
  \fi
  \ifdim\wd\@tempboxa=\z@ 
    \setbox\@ne\hbox to\columnwidth{\hss\kern-6pc\box\@ne\hss}%
  \else 
    \setbox\@ne\vbox{\unvbox\@tempboxa\parskip\z@skip
        \noindent\unhbox\@ne\advance\hsize-6pc\par}%
\fi
  \ifnum\@tempcnta<64 
    \addvspace\abovecaptionskip
    \moveright 3pc\box\@ne
  \else 
    \moveright 3pc\box\@ne
    \nobreak
    \vskip\belowcaptionskip
  \fi
\relax
}
\makeatother
\def\writefig#1 #2 #3 {\rlap{\kern #1 truecm
\raise #2 truecm \hbox{#3}}}

\newcommand{\be}{\mathbf{e}}

\title[Brownian motion in a Poissonian
potential with long range correlation]{Superdiffusivity for Brownian motion in a Poissonian
potential with long range correlation II:
upper bound on the volume exponent}

\author{Hubert Lacoin}
\address{CEREMADE, Place du Mar\'echal De Lattre De Tassigny
75775 PARIS CEDEX 16 - FRANCE}
\email{lacoin@ceremade.dauphine.fr}

\begin{document}

\maketitle
\selectlanguage{english}
\begin{poliabstract}{Abstract}
This paper continues a study on trajectories of Brownian Motion in a field of soft trap
whose radius distribution is unbounded. We show here that for both point-to-point and point-to-plane model
the volume exponent (the exponent associated to transversal fluctuation of the trajectories) $\xi$ is strictly less than $1$
 and give an explicit upper bound that depends on the parameters of the problem.
In some specific cases, this upper bound matches the lower bound proved in the first part of this work 
and we get the exact value of the volume exponent.
\end{poliabstract}

\selectlanguage{frenchb}
\begin{poliabstract}{Résumé}
Cet article est la seconde partie d'une étude sur les trajectoires Brownienne dans un champs de pièges mous dont 
le rayon est aléatoire et a une distribution non-bornée. 
Nous montrons que l'exposant de volume (qui est l'exposant associé aux fluctuations transversales des trajectoires) $\xi$
est strictement inférieur à $1$ et nous donnons une borne supérieure explicite qui dépend des paramètres du problème, et ceci aussi bien 
pour le modèle dans la configuration point-à-point que pour celui dans la configuration point à plan.
Dans certains cas particulier, cette borne supérieure coïncide avec la borne inférieure démontrée dans la première partie 
de cette étude, ce qui nous permets d'identifier la valeur de l'exposant de volume.\\
\textit{Mots clés: Polymères faiblements dirigés, Désordre gelé, Surdiffusivité, Mouvement brownien, Obstacles poissonien, 
Correlation.}
\end{poliabstract}

\selectlanguage{english}

\section{Introduction}

In this paper we investigate properties of the trajectories of Browian Motion in a disordered medium:
given a random function $V$ defined on $\bbR^d$ and $\gl> 0$,
we study trajectories of a Brownian motion $(B_t)_{t\ge 0}$
killed at (space-dependent) rate $\gl + V (B_t)$ conditioned to survive up to the hitting
time either of a distant hyperplane or a distant ball (we refer to these two cases respectively as point-to-plane and point-to-point).
We focus more specifically on transversal fluctuation, i.e.\ fluctuation of the trajectories along the directions that are normal to 
the line that links the two points.

\medskip

In an homogeneous medium these fluctuation are of order $\sqrt{L}$ where $L$ is the distance to the hyperplane or point.
It is commonly believed that disorder should make these fluctuation larger, \ e.g.\ of order $L^\xi$ where $\xi>1/2$ is called
volume exponent. This phenomenon is called superdiffusivity and should hold for low dimension ($d\le3$) or when amplitude of the variations of $V$ are large enough. 
We study it in a model where the random potential $V$ is generated by a field of soft trap of random IID radii.
The tail distribution of the radius of a trap is heavy-tailed so that our potential presents long range correlation.
This model is a variant of a more studied model of Brownian motion among soft obstacles extensively studied by Snitzman
(see the monograph \cite{cf:S} and reference therein) and for which superdiffusivity was shown to hold in dimension $2$ by Wüthrich ($\xi\ge 3/5$, \cite{cf:W2})
who also proved  a universal bound $\xi\le 3/4$, valid for any dimension \cite{cf:W3}.

\medskip

For our model with correlated potential, we proved in \cite{cf:H1} that superdiffusivity holds when $d=2$ and
in larger dimension when correlations in the environment are strong enough (see \eqref{fredh}). 
The lower-bound that we get for $\xi$ depends on the parameter
of the model and in certain cases it is larger than $3/4$ (which is an upper bound for the volume exponent in any dimension in a large variety of model
in the same universality class
see e.g. \cite{cf:Meja} for directed polymer, and \cite{cf:Lac} for directed Brownian Polymer in an environment with long-range transversal correlation).

\medskip

In this paper, our aim is to find an upper-bound for the volume exponent $\xi$.
It turns out that for some particular choices of the parameter, the upper bound one finds for $\xi$ matches the lower bound found in \cite{cf:H1} and
therefore allows us to derive the existence and exact value of the volume exponent (Corollary \ref{tyui}). 

\medskip

It is quite rare to be able to derive volume exponent for disordered model, even at the level of physicists prediction. 
For the two dimensional model studied in \cite{cf:W2} and a whole class of related random growth model (e.g.\ two dimensional first-passage percolation, 
oriented first-passage percolation and directed polymer in random environment in $1+1$ dimension) it is predicted that $\xi=2/3$ and it has
been proved in very particular cases (\cite{cf:J,cf:Sep,cf:BQS} and some more). These works have in common that they rely on exact calculation and 
therefore cannot be exported to general cases yet.

\medskip

Here, lower-bound and upper-bound are both derived using energy v.s.\ entropy comparisons, and
 the reason why we are able to get the exact exponent is 
somehow different. When the tail distribution of radiuses of traps gets heavy, most of the fluctuation are caused by very large traps
and this makes the system almost ``one-dimensional'' in a sense, and therefore easier to handle.

\section{Model and result}

Let $V^{\go}(x),x\in \bbR^d$ be a random potential defined as follows: we consider first a Poisson Point Process, in $\R^d \times \R^+$, viewed as a set of points 

\begin{equation}
 \go:= \{ (\go_i,r_i)\in \bbR^d\times \bbR_+ \ | \ i\in \N\} 
\end{equation}
(the ordering of the points $(\go_i,r_i)$ being made in some arbitrary deterministic way, e.g. such that $|\go_i|$ is an increasing sequence),
whose intensity is given by $\mathcal L \times \nu$ where $\mathcal L$ 
is the Lebesgue measure on $\bbR^d$ and $\nu$ is a probability measure on $\R^+$ . 
For the sake of simplicity we restrict to the case of $\nu$ satisfying
\begin{equation}
\forall r\ge 1,\  \nu([r,\infty])=r^{-\alpha},
\end{equation}
for some $\alpha>0$ (but the result would hold with more generality, e.g.\ assuming only that $\nu$ has power-law decay at infinity).
Denote by $\bbP$ and $\bbE$ the associated probability law and expectation.

\medskip
This process represents a field random traps centered at $\go_i$ of and radius $r_i$.
From $\go$ we construct the potential $V^\go: \bbR^d \rightarrow \bbR_+\cup \infty$ defined by
\begin{equation}
 V^{\go}(x):=\sum_{i=1}^{\infty} r_i^{-\gamma} \ind_{\{|x-\go_i|\le r_i\}},
\end{equation}
for some $\gamma>0$.
Note that  $V^{\go}(x)<\infty$ for every $x\in \bbR^d$, for almost every realization of $\go$ if and only if the condition
$\alpha+\gga-d>0$ holds. 
We suppose in what follows
that we have the stronger condition $\alpha-d>0$ (mainly not too have to treat too many different cases in the proof, but we could have results without this condition) 
which means that a given point lies almost surely in finitely many traps.

\medskip

Given $L>0$,
we consider the hyperplane $\mathcal H_L$ at distance $L$ from the origin.
\begin{equation}
\mathcal H_L:= \{L\}\times \bbR^{d-1}
\end{equation}
 Denote by $\bP$, $\bE$ (resp.\ $\bP_x$, $\bE_x$) the law and expectation associated 
to standard $d$-dimensional Brownian motion $(B_t)_{t\ge 0}$ started from the origin (resp.\ from $x$).

\medskip

Given $\gl>0$ we study the trajectories of a Brownian Motion started from the origin killed with rate $(V^{\go}(\cdot)+\gl)$ 
conditioned to survive till it hits $\mathcal H_L$.  
The survival probability is equal to 

\begin{equation}\label{defZ}
 Z_{L}^{\go}:= \bE\left[\exp\left(-\int_{0}^{T_{\mathcal H_L}} (V^{\go}(B_t)+\gl)\dd t\right)\right],
\end{equation}
(For any set $A$, $T_{A}$ denotes the hitting time of $A$).
The law of the trajectories conditioned to survival $\mu_L^{\go}$ is absolutely continuous with respect to $\bP$, and its density is  given by
\begin{equation}\label{defmu}
 \frac{\dd \mu_L^{\go}}{\dd \bP}(B):= 
\frac{1}{Z^\go_L}\exp\left(-\int_{0}^{T_{\mathcal H_L}}(V^{\go}(B_t)+\gl)\dd t\right) .
\end{equation}

To study transversal fluctuation of the trajectory around the 
axis $\bbR \be_1$ ($\be_1=(1,0,\dots,0)$ being the first coordinate vector),
one has to give a true definition to the notion of volume exponent 
discussed in the introduction.
In that aim, define

\begin{equation}
\mathcal C_L^{\xi}:= \{z\in \bbR^d \ | \ \exists \alpha\in[0,L],
 |z-\alpha \be_1|\le L^\xi \}=\bigcup_{\alpha\in[0,L]} B(\alpha \be_1, L^\xi).
\end{equation}
and 
\begin{equation}
\mathcal A_L^{\xi}:= \{ (B_t)_{t\ge0} \ | \forall s\in [0, T_{\mathcal H_L}], B_s\in \mathcal C_L^{\xi}\}.
\end{equation}
the event ``the trajectories stays in the tube $\mathcal C_L^{\xi}$ till the hitting time of $\mathcal H_L$''.
We define the upper and lower volume exponent $\xi_0$ and $\xi_1$ as follows:
 
\begin{equation}\begin{split}
\xi_1:=\inf\{\xi \ \ |  \  \lim_{L\to\infty}  \bbE\left[\mu_L^{\go}(\mathcal A_L^{\xi})\right]=1\},\\
\xi_0:=\sup\{\xi \ \ |  \  \lim_{L\to\infty}  \bbE\left[\mu_L^{\go}(\mathcal A_L^{\xi})\right]=0\}.
\end{split}\end{equation}
From the definition, $\xi_1\ge \xi_0$ but one expects that $\xi_1=\xi_0$ and their common value is 
referred to as the volume exponent.
The main result of this paper is to get an upper-bound on $\xi_1$.
Set
\begin{equation}
\tilde  \xi(\alpha,\gamma,d):=\max\left(\frac 3 4, \frac{1}{1+\alpha-d}, \min\left(\frac{1}{1+\gamma},\frac{2+d}{2\alpha}\right)\right)<1.
\end{equation}

\begin{theorem}\label{turiaf1}
For all $\xi> \tilde  \xi(\alpha,\gamma,d)$
One has 
\begin{equation}
 \lim_{L\to\infty}  \bbE\left[\mu_{L}^{\go}(\mathcal A_L^{\xi})\right]=1.
\end{equation}
Or equivalently $\xi_1\le \tilde \xi$.
\end{theorem}

In some special case, when $(\alpha-d)=\gamma\le 1/3$, then the upper bound above 
coincides with the lower-bound proved in a first study on this model \cite[Theorem 2.1]{cf:H1} 

\begin{equation}\label{fredh}
 \xi_0\ge \min\left(\frac{1}{2},\frac{1}{1+\alpha-d},\frac{3}{3+2\gamma+\alpha-d}\right),
\end{equation}
(to be more precise the definition of $\xi_0$ in \cite{cf:H1} is a bit different because the set $\mathcal C^\xi_L$ there is not exactly the same, but 
Theorem 2.1 there implies \eqref{fredh}).
And therefore

\begin{cor}\label{tyui}
For any value of $\alpha,d$ and $\gamma$ that satisfies
$(\alpha-d)=\gamma\le 1/3$
, one has
\begin{equation}
\xi_1=\xi_0=\frac{1}{1+\gamma}.
\end{equation}
\end{cor}

A much related problem is the the study of trajectories conditioned to survive up to the hitting time of a distant ball.
We introduce this model now for two reason:
\begin{itemize}
 \item  We use it as a tool for the proof of the result above.
 \item  An analogous result can be proved using the same method for this model.
\end{itemize}

For a Brownian Motion started at $x$ and killed with rate $\gl+V(\cdot)$, we denote by
\begin{equation}\label{defZ1}
Z^\go (x,y):= \bE_x \left[e^{-\int_0^{T_{B(y)}} (\gl+V(B_t)) \dd t }\ind_{\{T_{B(y)}<\infty\}}\right],
\end{equation}
the probability of survival up to the hitting time of $T_{B(y)}$ of $B(y)=B(y,1)$ the Euclidean ball of radius one and center
 $y$, $|x-y|\ge 1$ (we keep this notation for what follows and denote by
$B(z,r)$ the Euclidean ball of center $z$ radius $r\in \bbR_+$) , and by $\mu_{x,y}^\go$ the law of the trajectory
$(B_t)_{t\in[0,T_{B(y)}]}$ conditioned to survival, its derivative with respect to $\bP_x$ is equal to
\begin{equation}\label{defmu1}
\frac{\dd \mu_{x,y}^{\go}}{\dd \bP_x}:=\frac{1}{Z^\go (x,y)}e^{-\int_0^{T_{B(y)}} (\gl+V(B_t)) \dd t }\ind_{\{T_{B(y)}<\infty\}}.
\end{equation}
For this reason, for a given $y\in \bbR^d$ one defines in analogy with $\mathcal A_L^\xi$ and $\mathcal C_L^{\xi}$.
\begin{equation}
\mathcal C_y^{\xi}:= \left\{z\in \bbR^d \ | \ \exists \alpha\in [0,1],
 |z-\alpha y|\le |y|^\xi \right\}=\bigcup_{\alpha\in [0,1] } B(\alpha y, |y|^\xi).
\end{equation}
and 
\begin{equation}
\mathcal A_y^{\xi}:= \{ (B_t)_{t\ge0} \ \big| \forall s\in [0, T_{B(y)}], B_s\in \mathcal C_y^{\xi}\}.
\end{equation}
The following analogous of Theorem \ref{turiaf1}
\begin{theorem}\label{turiaf}
For all $\xi> \tilde  \xi(\alpha,\gamma,d)$,
one has 
\begin{equation}
 \lim_{|y|\to\infty}  \bbE\left[\mu_{0,y}^{\go}(\mathcal A_y^{\xi})\right]=1.
\end{equation}
\end{theorem}

\begin{rem}\rm For the point-to-point model, one does not have an equivalent of Corollary \ref{tyui}, the reason being that
the lower-bound that we have on $\xi_0$ in \cite{cf:H1} was slightly suboptimal. However we strongly believe that the analogous results hold. 
 \end{rem}

%
%

The ideas of this proof are inspired by \cite{cf:W3} and \cite{cf:SAOP} where an upper bound on the volume exponent is proved for 
model with traps of bounded range ( $\xi_1\le 3/4$). In \cite{cf:SAOP}, Sznitman uses martingale techniques to prove concentration of $Z^\go(x,y)$ around its mean, and in 
\cite{cf:W3}
Wüthrich uses these concentration results to prove the bound on the volume exponent.

\medskip

These techniques cannot directly apply to our model, and in fact both bounds proved in \cite{cf:SAOP} and \cite{cf:W3} 
do not hold when there are too strong correlations in the environment. This is not surprising as in \cite{cf:H1}
 it was shown that the upper-bound $\xi_1 \le 3/4$ proved in \cite{cf:W3} does not always hold.

Our strategy is to study the model with a slightly modified potential:
\begin{itemize}
\item First in Section \ref{modpot} we present our modification of the potential and show that 
it does not modify much that probabilities of $\mathcal A_y^{\xi}$, $\mathcal A_L^{\xi}$ (Proposition \ref{micmac}.
\item In Section \ref{concent}, we show that the partition function associated to the modified potential concentrates around its mean, using a multiscale analysis (Proposition
\ref{concentrat}).
\item In Section \ref{super}, we use Proposition \ref{concentrat} to prove Theorem \ref{turiaf1} and \ref{turiaf}.
\end{itemize}

\begin{rem}\rm Some of the refinement of the techniques (in particular, the multiscale analysis) used here are not needed if one simply wants to prove that 
that $\bbE\left[\mu_{0,y}^{\go}(\mathcal A_y^{\xi})\right]$ tends to zero for some $\xi<1$. The reason we use them is that they allow us to get a slightly
 better bound,
and that they are absolutely necessary to get Corollary \ref{tyui}.

\end{rem}

\section{Modification of the potential $V$}\label{modpot}

We slightly modify   $V$ in order to have a potential with nicer properties.
In particular we want to
\begin{itemize}
\item Make it bounded (by a constant depending on $L$)
\item Suppress traps whose radius is too large to have only finite range correlation.
(what ``too large'' depends also on $L$) in order to treat potential for far away region independently.
\end{itemize}
%
%
%
%

In this section we define this modified  potential and show that with our choice for modifications of the potential does not significantly change the probability of $\mathcal A_L^{\xi}$
(or if it does, that it does it in the right direction).
Given $\xi>\tilde \xi(\alpha, t, d)$ we define 
\begin{equation}
\bar\xi:=\min(\xi,d/\alpha).
\end{equation}
The modified potential $\bar V^{\go}_L$ by

\begin{equation}
\bar V^{\go}_L(x):=\sum_{n=0}^{\bar \xi \log_2 L} 
\min\left(\sum_{i=1}^\infty \ind_{\{r_i\in[2^n, 2^{n+1})\}} 
r_i^{-\gga} \ind_{\{| x-\go_i | \le r_i\}},2^{-n\gamma}\log L\right).
\end{equation}
(it is the same as $V$ except that it ignores traps whose radius is larger than $2L^{\bar\xi}$, and that it cuts the contribution of 
traps of  diameter $[2^n, 2^{n+1})$  at the level $2^{-n\gamma}\log L$).
In analogy with \eqref{defZ}, \eqref{defmu}, \eqref{defZ1} \eqref{defmu1} one defines $\bar Z^\go_L$, 
$\bar\mu_L^\go$, $\bar Z^\go(x,y)$, $\bar\mu_L^\go(x,y)$,
by replacing $V^\go$ by $\bar V_L^{\go}$.

\medskip

This is not a very drastic modification and it should not change the probability of  $\mathcal A_L^\xi$ (and that of  $\mathcal A_y^\xi$ for $|y|=L$) 
and for two reasons:
\begin{itemize}
\item With $\bbP$-probability going to one, there is no trap of radius more than $2L^{\bar \xi}$ that intersects $\mathcal C_L^\xi$.

\item With $\bbP$-probability going to one, $$\max\left(\sum_{i=1}^\infty \ind_{\{r_i\in[2^n, 2^{n+1})\}} 
r_i^{-\gga} \ind_{\{| x-\go_i | \le r_i\}},2^{-n\gamma}\log L\right)$$ is equal to $\sum_{i=1}^\infty \ind_{\{r_i\in[2^n, 2^{n+1})\}} 
r_i^{-\gga} \ind_{\{| x-\go_i | \le r_i\}}$ for all $x$ in $B(0,L^2)$ the Euclidean ball of radius $L^2$ centered at zero.
\end{itemize}
And indeed one has

\begin{proposition}\label{micmac}
 There exists  $c$
such that, for all $\xi\ge \tilde \xi$, for any $y$ such that $|y|=L$,  with probability going to one when $L$ tends to infinity, 
\begin{equation}\begin{split}\label{bigmac}
\mu^\go_{0,y}(\mathcal A^\xi_y)\ge \bar\mu^\go_{0,y}(\mathcal A^\xi_y)-e^{-cL^2}, \\
\mu_L^\go(\mathcal A^\xi_L)\ge \bar\mu_L^\go(\mathcal A^\xi_L)-e^{-cL^2}.
\end{split}\end{equation}
\end{proposition}

\begin{proof}
We only prove the first line in \eqref{bigmac} which is the result concerning the point-to-point model. The other one is proved analogously.
Set 

\begin{equation}
\tilde V^{\go}_L(x):=
\sum_{i=1}^\infty \ind_{\{r_i\le 2 L^{\bar\xi} \}} 
r_i^{-\gga} \ind_{\{| x-\go_i | \le r_i\}}.
\end{equation}
(the only difference with $V$ is that traps with radius larger than $2L^{\bar \xi}$ are not taken into account)
and define $\tilde \mu^\go_{0,y}$ and $\tilde Z^{\go}(0,y)$ as in \eqref{defZ1} and \eqref{defmu1}.

\medskip

Our first job is to show that $\bar \mu^\go_{0,y}$ , and $\tilde \mu^\go_{0,y}$ are close in total variation, 
then we compare $\tilde \mu^\go_{0,y}(\mathcal A_y^\xi)$ with $\bar\mu^\go_{0,y}(\mathcal A_y^\xi)$.
We notice that  $\bar V^\go$ and $\tilde V^\go$ coincide with probability tending
to one on $B(0,L^2)$, indeed a consequence of Lemma \ref{techin} (proved in the appendix)
is that 

\begin{equation}
\bbP\left[\exists x \in B(0,L^2), \bar V^\go(x)\ne \tilde V^\go(x) \right]\le \frac{1}{L}.
\end{equation}

When the event $\{\forall  x \in B(0,L^2), \bar V^\go(x)= \tilde V^\go(x) \}$ holds
then 
 $\bar \mu^\go_{0,y}(\cdot \ |\mathcal S_L)$ and $\tilde \mu_{0,y}(\cdot \ | \mathcal S_L)$, the measures conditioned on the event
\begin{equation}
\mathcal S_L= \{ \ \forall t \in[0,T_{B(y)}] \  , \ |B_t|\le L^2 \ \},
\end{equation}
are equal, and therefore it remains only to show that with large probability $\bar \mu_{0,y}((\mathcal S_L)^c)$ 
and $\tilde \mu_{0,y}((\mathcal S_L)^c)$ are small. Set $\tau_{L^2}:=\inf\{t, |B_t|\ge L^2\}$, then

\begin{equation}\label{retac}
\bar \mu_{0,y}((\mathcal S_L)^c)\le \frac{\bE\left[e^{-\gl T_{B(y)}}\ind_{\{\tau_{L^2}\le T_{B(y)}<\infty \}}\right]}{\bar Z(0,y)}\le 
\frac{\bE\left[e^{-\gl \tau_{L^2}}\right]}{\bar Z(0,y)}.
\end{equation}
As $\bar V(x)\le \log L$ for all $x$, thanks to standard tubular estimate for Brownian motion  (see e.g. (1.11) of \cite{cf:SCP}) for $C$ large enough
\begin{equation}
\log \bar Z^\go(0,y)\ge -C L\log L.
\end{equation}
Other standard estimates give that there exists $c$ such that
\begin{equation}
\bE\left[e^{-\gl \tau_{L^2}}\right]\le e^{-cL^2}
\end{equation}
so that for $L$ large 
\begin{equation}\label{retic}
\bar \mu^\go_{0,y}((\mathcal S_L)^c)\le e^{-cL^2}
\end{equation} 
tends to zero when $L$ tends to infinity.
Working on the event ``$\bar V$ and $\tilde V$ coincide on $B(0,L^2)$'' holds we get the same conclusion for $\tilde \mu_{0,y}^{\go}$
so that with probability going to one

\begin{equation}\label{TV}
\|\tilde \mu_{0,y}^{\go}-\bar \mu^\go_{0,y}\|_{TV}\le e^{-cL^2}.
\end{equation}
Now we remark that with probability going to one
$\tilde V$ and $V$ coincide on $\mathcal C_y^\xi$ i.e.\ that 

\begin{multline}
\lim_{|y|=L\to \infty}\bbP\left[\exists x \in \mathcal C_y^\xi, \bar V^\go(x)\ne  V^\go(x) \right]\\
= \lim_{|y|=L\to \infty}\bbP\left[\exists i, r_i\ge 2L^{\bar \xi}, B(\go,r_i)\cap \mathcal C_y^\xi\ne \emptyset  \right]=0.
\end{multline}
Indeed the number of traps of radius larger that $2L^{\bar \xi}$ that intersects $\mathcal C_y^\xi$ is a Poisson variable and its mean is 

\begin{equation}\label{zer}
\int_{2L^{\bar \xi}}^{\infty}  \left(\sigma_d (r+L^{\xi})^d+L\sigma_{d-1} (r+L^\xi)^{d-1}\right)\alpha r^{-\alpha-1}\dd r \le C L^{1+\xi(d-1)-\alpha\bar \xi}. 
\end{equation}
We let the reader check that with our choice of $ \xi$ and $\bar \xi$, 
\begin{equation}
1+\xi(d-1)-\alpha\bar \xi < 0
\end{equation}
so that the r.h.s of \eqref{zer} tends to zero.
For any $x$ in $(\mathcal C_y^\xi)^c$ we necessarily have $\tilde V^\go(x)\le V^\go(x)$ from the definitions, so that on the event ``$\tilde V$ and $V$ coincide on $\mathcal C_y^\xi$",

\begin{equation}
 \mu_{0,y}^{\go}(\mathcal A_y^\xi)\ge \tilde \mu_{0,y}^{\go}(\mathcal A_y^\xi).
\end{equation}
A combination of the above and \eqref{TV} allows us to conclude.

\end{proof}

\section{Concentration inequalities}\label{concent}

In this section, one derives some concentration inequalities similar to the one obtained in \cite{cf:SAOP}
for the log partition function with the modified potential $\log \bar Z^\go(u,v)$. It could be shown that for some choice of parameters,
these concentration results do not hold for the original potential.
We suppose that $L$ is fixed, and set

\begin{equation}
\bar Z^\go (u,v):= \bE_x \left[e^{\int_0^{T_y} (\gl+\bar V^\go(B_t)) \dd t }\right].
\end{equation}

\begin{equation}
\chi=\chi(\xi):=\max\left(\frac 1 2, (1-\gamma)\bar \xi, \frac{1}{2}(1+\bar \xi(1+d-2\gamma-\alpha))\right).
\end{equation}

\begin{proposition}\label{concentrat}
Suppose that $\xi\ge \tilde \xi(\alpha,\gamma,d)$
For any $\gep>0$ one can find $\gd$ such that
for any $(u,v)\in \bbR^d$, $|u-v|\le 2L$
\begin{equation}
\bbP\left( |\log \bar Z^\go (u,v)-\bbE \log \bar Z^\go (u,v)| \ge L^{\chi+\gep}\right)\le \exp(-L^{\delta}).
\end{equation}
\end{proposition}

As the environment is translation invariant, we need only  to prove the result the case $|v|\le 2L$, $u=0$.

The proof of this proposition requires a multi-scale analysis, to treat traps of different scale in separate steps.
One could get a result by doing a rougher analysis, but this would never get us something optimal.
On the contrary, the multi-scale analysis allows us to get sharper results that are optimal for some special choice of the parameters
(i.e.\ they allow to get an upper bound on the volume exponent that matches the lower bound).

\medskip

For all $n$ define $\mathcal F_n$ to be the sigma-algebra generated by the traps of radius smaller than $2^n$

\begin{equation}
\mathcal F_n:=\sigma\left( \go(A), A \in \mathcal B(\bbR^d\times \bbR_+),  A\subset \bbR^d\times [1,2^{n}]\right),
\end{equation}
($\go(A)$ above stands for the number of point in $A$
and $\mathcal B(\bbR^d\times \bbR_+)$ stands for the the sigma fields of Borel-sets).
We define for $n\ge 0$,

\begin{equation}
M_n:=\bbE\left[\log \bar Z^{\go} (0,v) | \mathcal F_n\right],
\end{equation}
(Note that $M_{\bar\xi\log_2 L}=\log \bar Z^{\go} (0,v)$).
The sequence $(M_n)_{n\ge 0}$ is a martingale for the filtration $(\mathcal F_n)_{n\ge 0}$.
We prove Proposition \ref{concentrat} by proving concentration for every increment of
$(M_n)_{n\ge 0}$ (there are only $O(\log L)$ increments so that this is sufficient to
get the result).

\begin{lemma}\label{trebo}
For any $\gep$ there exists $\delta$ such that
for all $n\in [1, \bar\xi\log_2 L]$,
\begin{equation}
\bP\left[| M_{n}- M_{n-1}|\ge 
L^{\chi+\gep}\right]\le e^{-L^{\delta}}.
\end{equation}
\end{lemma}

To prove the above Lemma we can adapt and use the technique developed in \cite{cf:SAOP}:
given $n$ we partition $\mathbb \bbR^d$ in disjoint cubes of side length $2^n$,

\begin{equation}
(2^n x +[0,2^n]^d)_x\in \bbZ^d,
\end{equation}
and index them by $\N$ in an arbitrary way and call that sequence $(C_{n,k})_{k\ge1}$.
Then one sets
\begin{equation}
\mathcal F_{n,k}:=\sigma\left( \go(A), A \in \mathcal B(\bbR^d\times \bbR_+), 
A\subset \left[(\bbR^d\times [1,2^{n-1}])\cup  (\bigcup_{i=1}^k C_{n,i}\times\left[2^{n-1},2^n\right]) \right]\right),
\end{equation}
which is the sigma algebra generated by traps of radius smaller than $2^{n-1}$ and
traps of radius in $[2^{n-1},2^n]$ whose centers are located in the set of cube $\bigcup_{i=1}^k C_{n,i}$ ($\go(A)$ above stands for the number of point in $A$
and $\mathcal B(\bbR^d\times \bbR_+)$ stands for the the sigma field of Borel-sets).

\medskip

One defines for $k\ge 0$
\begin{equation}
M_{n,k}:= \bbE\left[ \log \bar Z^{\go}(0,v)\ | \ \mathcal F_{n,k} \right].
\end{equation}

One remarks that for fixed $n$, $(M_{n,k})_{k\ge 0}$ is a martingale for the filtration $(\mathcal F_{n,k})_{k\ge 0}$.
It is an interpolation between $M_{n-1}=M_{n,0}$ and $M_n=M_{n,\infty}$.
This allows us to use a Martingale concentration result by Kesten   to prove Lemma \ref{trebo}.

\begin{proposition}
 
( From \cite[Theorem 3]{cf:K}\label{Keste} )

Let $(X_n)_{n\ge 0}$ be a martingale with respect to the filtration $\mathcal G_n$, (law $P$ expectation $E$)
that satisfies
\begin{equation}
|X_{n+1}-X_n|\le c_1, \forall n\ge \bbN.
\end{equation}
and 
\begin{equation}
E\left[ (X_{n+1}-X_n)^2 \  | \ \mathcal G_n \right] \le  E\left[ V_n \  | \ \mathcal G_n \right] 
\end{equation}
for some sequence of random variable $(V_n)_{n\ge 0}$ satisfying:
\begin{equation}
P(\sum_{n\ge0} V_n \ge x)\le e^{-c_2 x}.
\end{equation}
for all $x\ge c_3$.
\medskip

Set $x_0:= \max(\sqrt{c_3},c_1)$. Then $X_{\infty}=\lim_{n\to \infty} X_n$ exists and for all $x\le c_2x_0^{3}$
\begin{equation}
P(|X_{\infty}-X_0|\ge x)\le C\left(1+\frac{1}{c_2x_0}\right)e^{- \frac{x}{Cx_0}},
\end{equation}
where $C$ is a universal constant not depending on the $(c_i)_{i=1}^3$.
\end{proposition}

Our proof of Lemma \ref{trebo} consists simply in checking, for each value of $n$, the assumptions  Proposition \ref{Keste} for the martingale
$(M_{n,k})_{k\ge 0}$.
For any cube $C_{n,k}$ one defines 

\begin{equation}
\tilde C_{n,k}:= \bigcup_{x\in C_{n,k}}  B(x, 2^n).
\end{equation}
(this is the zone where the $V$ can be modified when one adds traps of radius smaller than $2^n$ with center in
$C_{n,k}$)
 and $T_k$ to be  the hitting time of $\tilde C_{n,k}$.

\begin{lemma}\label{dacodac}
For every $n$ and $k$ set $\gD M_{n,k}=M_{n,k}-M_{n,k-1}$. 
One can find a constant $C$ such that for every $n$ and $k$,
\begin{equation}\label{borbro}
|\gD M_{n,k}|\le C(\log L)^2 2^{n(1-\gamma)}.
\end{equation}
and
\begin{equation}\label{truy}
\bbE\left[ |\gD M_{n,k}|^2 | \mathcal F_{n,k-1} \right]\le  \bbE\left[ U_{n,k} | \mathcal F_{n,k-1} \right]
\end{equation}
where
\begin{equation}\label{truy2}
U_{n,k}= C
\bar \mu^\go_{0,y}\left( T_{k} \le T_{B(y)}\right)  2^{n\left[2(1-\gamma)+d-\alpha\right]}(\log L)^2.
\end{equation}
Moreover
\begin{equation}\label{truy3}
\bbP\left[\sum_{k=1}^\infty U_{n,k}\ge x\right]\le e^{- C^{-2}2^{n(-1+2\gamma+\alpha-d)}(\log L)^{-2} x}
\end{equation}
for all $x\ge C^2 L 2^{n(1-2\gamma+d-\alpha)}(\log L)^3$.
\end{lemma}

\begin{proof}[Proof of Lemma \ref{trebo}]
According to Lemma \ref{dacodac} the assumptions of Proposition  \ref{Keste} are satisfied with
\begin{equation}\begin{split}
c_1=c_1(L,n)&:= C(\log L)^2 2^{n(1-\gamma)}\le C(\log L)^2 L^{\bar \xi(1-\gamma)_+},\\
c_2=c_2(L,n)&:=C^{-2}2^{n(-1+2\gamma+\alpha-d)}(\log L)^{-2}\ge C^{-2}L^{-\bar\xi(1-2\gamma+d-\alpha)_+}(\log L)^{-2},\\
c_3=c_3(L,n)&:= C^2 L 2^{n(1-2\gamma+d-\alpha)}(\log L)^3\le C^2 (\log L)^3L^{1+\bar \xi(1-2\gamma+d-\alpha)_+}.
\end{split}\end{equation}
And therefore we get that for  $x_0(L)=C(\log L)^2 L^{\max \left(\bar \xi(1-\gamma)_+, \frac{1+\bar \xi(1-2\gamma+d-\alpha)_+}{2}\right)}$,
for all $t\le c_2x_0^2/C$ (and note that $c_2x_0^2\ge L$)

\begin{equation}
\bbP\left[|M_{n-1}-M_n|\ge Cx_0 t  \right]\le (1+L^d)e^{-t}.
\end{equation}
provided the constance $C$ has been chosen large enough.
This is enough to conclude.
\end{proof}

At this point of the proof, we can explain a bit better our choice for the multi-scale analysis, and for the modification of the potential.
Both are aimed to optimize the constant $c_1$, $c_2$, $c_3$ above.
 
\begin{proof}[Proof of Lemma \ref{dacodac}]
One defines $\tilde \go$ , to be an independent copy of the environment $\go$ (let its law be denoted by $\tilde \bbE$).
Let $\go_{n,k}$  be an interpolation between $\go$ and $\tilde \go$ defined by
\begin{multline}
\go_{n,k}:=\left\{ (\go_i,r_i)\ \big| \ (\go_i,r_i)\in (\bbR^d\times [1,2^{n-1}])\cup  (\bigcup_{j=1}^k C_{n,j}\times\left[2^{n-1},2^n\right])\right\}\\ 
\cup \left\{ (\tilde \go_i,r_i) \ \big| \ 
(\go_i,r_i)\in (\bigcup_{j=k+1}^\infty C_{n,j}\times\left[2^{n-1},2^n\right])\cup  (\bbR^d\times [2^{n},\infty)) \right\}.
\end{multline}
And set 
\begin{equation}\begin{split}
V_{n,k}&:=\bar V^{\go_{n,k}},\\
Z_{n,k}(u,v)&:= \bar Z^{\go_{n,k}}(u,v)\\
\mu^{n,k}_{u,v}&:=\bar \mu^{\go_{n,k}}_{u,v}
\end{split}\end{equation} 
With this notation, note that $\go_{n,k}$ has the same distribution as $\go$ and that
\begin{equation}
M_{n,k}=  \tilde \bbE\left[ \log  Z_{n,k}(0,v)\right].
\end{equation} 
Furthermore
\begin{equation} \label{ledelta}
|\gD M_{n,k}|\le \tilde \bbE \left[\log \max \left( \frac{Z_{n,k-1}}{Z_{n,k}}(0,y),\frac{Z_{n,k}}{Z_{n,k-1}}(0,y)\right) \right]
\end{equation}

The first step of our proof is to bound 
$ \frac{Z_{n,k-1}}{Z_{n,k}}(0,y)$ and $\frac{Z_{n,k}}{Z_{n,k-1}}(0,y)$ by simpler functional depending only on 
$V_{n,k}$, $V_{n,k-1}$ in $\tilde C_k$. We use the following (abuse of) notation 

\begin{equation}
(V_{n,k}-V_{n,k-1})_+:= \max_{x\in \bbR^d} (V_{n,k}(x)-V_{n,k-1}(x))_+
\end{equation}

\begin{lemma}\label{trumu}
There exists a constant $C$ such that for all $n$ and $k$, for all $v$ in $\bbR^d$
\begin{equation}\begin{split}
\frac{Z_{n,k-1}}{Z_{n,k}}(0,v)
&\le 1+\mu^{n,k}_{0,v}\left[ T_k \le T_{B(v)}\right] (e^{C(V_{n,k}-V_{n,k-1})_+ 2^n \log L}-1),\\
\frac{Z_{n,k}}{Z_{n,k-1}}(0,v)&\le1+ \mu^{n,k-1}_{0,v}\left[ T_k \le T_{B(v)}\right] (e^{C(V_{n,k-1}-V_{n,k})_+ 2^n \log L}-1).
\end{split}\end{equation}
\end{lemma}

\begin{proof}

By symmetry of the problem it is sufficient to show that 

\begin{equation}
\frac{Z_{n,k-1}}{Z_{n,k}}(0,v)-1\le \mu^{n,k}_{0,v}\left[ T_k \le T_{B(v)}\right] (e^{C(V_{n,k}-V_{n,k-1})_+ 2^n \log L}-1)
\end{equation}
Using the Markov property at $T_k$ one gets

\begin{equation}
\frac{Z_{n,k-1}}{Z_{n,k}}(0,v)=\mu^{n,k}_{0,v}\left[T_{k} > T_{B(v)}\right]+\mu^{n,k}_{0,v}\left[ T_k\le T_{B(v)}\ ; \ \frac{Z_{n,k-1}}{Z_{n,k}}(B_{T_k},v)\right],
\end{equation}
and hence
\begin{multline}
\frac{Z_{n,k-1}}{Z_{n,k}}(0,v)-1=\mu^{n,k}_{0,v}\left[ T_k\le  T_{B(v)}\ ; \ \left(\frac{Z_{n,k-1}}{Z_{n,k}}(B_{T_k},v)-1\right)\right]\\
\le \mu^{n,k}_{0,v}\left[ T_k \le T_{B(v)}\right] \max_{z\in \partial \tilde C_k}\left(\frac{Z_{n,k-1}}{Z_{n,k}}(z,v)-1\right).
\end{multline}
We are left with showing that for all $z\in \partial \tilde C_k$

\begin{equation}
\frac{Z_{n,k-1}}{Z_{n,k}}(z,v)\le e^{C(V_{n,k}-V_{n,k-1})_+ 2^n \log L}.
\end{equation}
One has
\begin{multline}
\frac{Z_{n,k-1}}{Z_{n,k}}(z,v)=\mu^{n,k}_{z,v}\left[e^{\int_{0}^{T_{B(v)}} (V_{n,k}-V_{n,k-1})(B_t)\dd t}\right]\\
\le
\mu^{n,k}_{z,v}\left[e^{\int_{0}^{T_{B(v)}} (V_{n,k}-V_{n,k-1})_+(B_t)\dd t}\right].
\end{multline}
We study the tail distribution of the variable $\int_{0}^{T_{B(v)}} (V_{n,k}-V_{n,k-1})_+(B_t)\dd t$ under $\mu^{n,k}_{z,v}$.
On the event $\int_{0}^{T_{B(v)}} (V_{n,k}-V_{n,k-1})_+(B_t)\dd t>a$
one can define 
\begin{equation}
\tau_a:= \min\left\{ t>0\  |\  \int_{0}^{t} (V_{n,k}-V_{n,k-1})_+(B_s)\dd s=a\right\}.
\end{equation}
Necessarily, (as $V_{n,k}-V_{n,k-1}\equiv 0$ outside of $\tilde C_{n,k}$)
\begin{equation}
\tau_a\ge \frac{a}{ (V_{n,k}-V_{n,k-1})_+} \text{ and } B_{\tau_a}\in \tilde C_{n,k}.
\end{equation}

Using the Markov property and the above one gets that 

\begin{multline}
\mu^{n,k}_{z,v}\left(\int_{0}^{T_{B(v)}}\left( (V_{n,k}-V_{n,k-1})_+(B_t)\dd t\ge a\right)\right)\\
=\frac{1}{Z_{n,k}(z,v)}
\bE_z \left[ e^{-\int_0^{\tau_a}( \gl+V_{k}(B_t))\dd t} Z_{n,k}(B_{\tau_a},v)\right]\\
\\ \le \frac{1}{Z_{n,k}(z,v)}
\bE_z \left[ e^{-\gl \tau_a-a} Z_{n,k}(B_{\tau_T},v)\right]\\
\le e^{-\left(\frac{\gl}{ (V_{n,k}-V_{n,k-1})_+}+1\right)a} \max_{x\in \tilde C_k} \frac{Z_{n,k}(x,v)} {Z_{n,k}(z,v)}
\le  e^{-\left(\frac{\gl}{ (V_{n,k}-V_{n,k-1})_+}+1\right)a +c 2^n\log L},
\end{multline}
where in the last inequality one used an Harnack-type inequality (it is proved in (2.22) pp. 225 in \cite{cf:S} for $x$ and $z$ such that $|x-z|\le 1$ so that we can get the result below by iterating it) there exists a constant $c$ such that :
\begin{equation}
\forall x \forall z \in \bbR^d, \quad \left|\log  \frac{Z_{n,k-1}(x,v)} {Z_{n,k-1}(z,v)}\right|\le c (1+|x-z|) \| V_{n,k-1} \|_\infty.
\end{equation}
Hence
\begin{multline}
\mu^{n,k}_{z,v}\left[e^{\int_{0}^{T_{B(v)}} (V_{n,k}-V_{n,k-1})_+(B_t)\dd t}\right]\\
\le 1+ \int_0^{\infty}  e^{a} \min\left(1,e^{ -\left(\frac{\gl}{ (V_{n,k}-V_{n,k-1})_+}+1\right)a+c 2^n\log L} \right)\dd a\\
\\= \frac{\gl+(V_{n,k}-V_{n,k-1})_+}{\gl} e^{\frac{c 2^n\log L(V_{n,k}-V_{n,k-1})_+}{\gl+(V_{n,k}-V_{n,k-1})_+}}.
\end{multline}

\end{proof}

Let us introduce the notation
\begin{equation}\begin{split}
\mathcal N_{n,k,+}&:=|\{\text{ points  that are in $\go_{n,k}$ and not in $\go_{n,k-1}$}\}|,\\
\mathcal N_{n,k,-}&:=|\{\text{ points  that are in $\go_{n,k-1}$ and not in $\go_{n,k}$}\}|.
\end{split}\end{equation}
These two quantities are independent Poisson variable of mean 
$2^{n(d-\alpha)}(2^\alpha-1)$.
According to the definition of $\bar V^\go$ and $V_{n,k}$ one has
\begin{equation}\begin{split} \label{rzez}
(V_{n,k}-V_{n,k-1})_+&\le 2^{-(n-1) \gga} (\mathcal N_{n,k,+}\vee \log L),\\
(V_{n,k-1}-V_{n,k})_+&\le 2^{-(n-1) \gga} (\mathcal N_{n,k,-}\vee \log L).
\end{split}\end{equation}
Combining Lemma \ref{trumu}, equations \eqref{ledelta} and \eqref{rzez}, one gets \eqref{borbro}.
In order to get \eqref{truy} we 
use equation \eqref{ledelta}  to get that
\begin{equation}
|\gD M_k|^2\le \tilde \bbE \left[ \max \left( \left(\log \frac{Z_{n,k-1}}{Z_{n,k}}(0,v)\right)^2,\left(\log \frac{Z_{n,k}}{Z_{n,k-1}}(0,v)\right)^2 \right)\right].
\end{equation}
And from Lemma \ref{trumu} ,

\begin{multline}
\log \frac{Z_{n,k-1}}{Z_{n,k}}(0,v)\le \log \left(1+\mu^{n,k}_{0,v}(T_k \le T_{B(v)})e^{C 2^n (\log L)(V_{n,k-1}-V_{n,k-1})_+}\right)\\
\stackrel{\text{(Jensen)}}{\le} C 2^n (\log L)\mu^{n,k}_{0,v}(T_k \le T_{B(v)})(V_{n,k-1}-V_{n,k-1})_+\\
\stackrel{\eqref{rzez}}{\le}C 2^{n(1-\gamma)} (\log L) \mu^{n,k}_{0,v}(T_k \le T_{B(v)}) \mathcal N_{n,k,+}.
\end{multline}
One can get an analogous bound for $\log \frac{Z_{n,k}}{Z_{n,k-1}}(0,v)$ and get that 

\begin{multline}\label{piou11}
|\gD M_k|^2\le C^2 4^{n(1-\gamma)} (\log L)^2\\
\times \tilde \bbE \left[ \max\left(\mathcal N_{n,k,+}^2\mu^{n,k}_{0,v}(T_k \le T_{B(v)})^2,\mathcal N_{n,k,-}^2\mu^{n,k-1}_{0,v}(T_k \le T_{B(v)})^2\right)\right].
\end{multline}
Replacing $\max$ by a sum and conditioning to $\mathcal F_{n,k-1}$ one gets that to bound $\bbE\left[|\gD M_k|^2\ | \ \mathcal F_{n,k-1}\right]$
it is sufficient to bound 
\begin{equation}\label{piou0}
 \begin{split}
  &\bbE\left[\tilde \bbE \left[\mathcal N_{n,k,+}^2\mu^{n,k}_{0,v}(T_k \le T_{B(v)})^2\right] \ | \ \mathcal F_{n,k-1}\right],\\
  &\bbE\left[\tilde \bbE \left[\mathcal N_{n,k,-}^2\mu^{n,k-1}_{0,v}(T_k \le T_{B(v)})^2\right] \ | \ \mathcal F_{n,k-1}\right].
 \end{split}
\end{equation}
The reader can check that 
\begin{equation}\begin{split}
\tilde \bbE \left[\mathcal N_{n,k,+}^2\mu^{n,k}_{0,v}(T_k \le T_{B(v)})^2\right]&=\bbE \left[\mathcal N_{n,k,+}^2\bar\mu^{\go}_{0,v}(T_k \le T_{B(v)})^2  \ | \ \mathcal F_{n,k}\right],\\
\tilde \bbE \left[\mathcal N_{n,k,-}^2\mu^{n,k-1}_{0,v}(T_k \le T_{B(v)})^2\right]&=\bbE \left[\mathcal N_{n,k,+}^2\bar\mu^{\go}_{0,v}(T_k \le T_{B(v)})^2  \ | \ \mathcal F_{n,k-1}\right].
\end{split}\end{equation}
and thus we have just to bound from above control the r.h.s\ of the second line.
We rewrite it as follows
\begin{multline}\label{piou}
\bbE\left[  \bar \mu^{\go}_{0,v} \left(T_k \le T_{B(v)}\right)^2 \mathcal N^2_{n,k,+} \ | \ \mathcal F_{n,k-1}\right]\\
=\bbE\left[  \bbE\left[\bar\mu^{\go}_{0,v} \left(T_k \le T_{B(v)}\right)^2 \ | \ \mathcal F_{n,k-1} \vee \sigma (\mathcal N_{n,k,+})\right]\mathcal N^2_{n,k,+} \ | \ \mathcal F_{n,k-1}\right]
\end{multline}
Then one can remark that
$$\bbE\left[\bar\mu^{\go}_{0,v} \left(T_k \le T_{B(v)}\right)^2 \ | \ \mathcal F_{n,k-1}\vee \sigma (\mathcal N_{n,k,+})\right] $$
is a non-increasing function of $\mathcal N_{n,k,+}$.
If $f$ is a non-increasing function of $\mathcal N$, $g$ a non-decreasing function of $\mathcal N$ then
\begin{equation}
\bbE\left[f(\mathcal N)g(\mathcal N)\right]\le \bbE\left[f(\mathcal N)\right] \bbE\left[ g(\mathcal N)\right]
\end{equation}
Therefore the right hand-side of \eqref{piou} is less than 
\begin{multline}
\bbE\left[\bbE\left[\mathcal N^2_{n,k,+}\right] \bar \mu^{\go}_{0,v} \left(T_k \le T_{B(v)}\right)^2 \ | \ \mathcal F_{n,k-1} \right]\\
= \bbE\left[\mathcal N^2_{n,k,+}\right]^2
\bbE\left[\bar \mu^\go_{0,v} \left(T_k \le T_{B(v)}\right)^2 \ | \ \mathcal F_{n,k-1} \right]\\
\le C 2^{n(d-\alpha)}\bbE\left[\bar \mu^\go_{0,v} \left(T_k \le T_{B(v)}\right)\ | \ \mathcal F_{n,k-1} \right] .
\end{multline}
which combined with \eqref{piou11}, \eqref{piou0} and \eqref{piou} ends  the proof of \eqref{truy}-\eqref{truy2}.

%
As for \eqref{truy3}, notice that 
\begin{equation}
\sum_{k=1}^\infty \bar \mu^\go_{0,v} \left(T_k \le T_{B(v)}\right)=\bar \mu^\go_{0,v} \left(A_{T_{B(y)}}\right),
\end{equation}
where
\begin{equation}
A_T:=| \{ x\in \Z^d \ | \ \tilde C_x\cap \{ B_t\ ,\ t\in [0,T]\}\ne \emptyset \}|.
\end{equation}
denotes the number of different $\tilde C_x$ visited before $T$.
Large deviation estimates for the upper-tail distribution of  $A_{T_{B(y)}}$ under $\bar \mu^\go_{0,v}$ are computed 
in the appendix (Lemma \ref{tuture})
and they allow us to obtain \eqref{truy3}.

\end{proof}

\section{Volume exponent from fluctuation}

\label{super}
\subsection{Preliminary result}

Before going in to the proof of Theorem  \ref{turiaf1} and \ref{turiaf}, we need a result that controls
the growth of the expected value of $\log \bar Z^\go(0,y)$ as a function of $|y|$.

Set $y_r:=(r,0,\dots,0)$ and define
\begin{equation}
\alpha(r):= -\bbE\left[ \log  \bar Z^{\go}(0,y_r) \right], 
 \end{equation}

It is natural to think that $r\mapsto \alpha(r)$ is increasing function of $r$ and that its growth is linear, but we cannot prove it.
Instead we prove a weaker result that will be sufficient to our purpose.

\begin{lemma}\label{witchi}
There exists a constant $c=c(\gl)$ such that for any $l\ge L^{\chi+\gep}$, $r\le 2L$ one has, for all large enough $L$,
\begin{equation}
\alpha(r+l)\ge \alpha(r)+  c l.
\end{equation}

\end{lemma}
 
 \begin{proof}

Let us consider a family of ball $(B(x_i,1))_{i\in \{1,\dots, k_r\}}$, $x_i\in \partial B(0,r)$ with 
$k_r=O(r^{d-1})$ that cover the sphere $\partial B(0,r)$,
\begin{equation}\label{coc}
\partial B(0,r) \subset \bigcup_{i=1}^ {k_r}(B(x_i,1)).
\end{equation}
In order to reach $y_{r+l}$ starting from zero,
 a Brownian motion has to touch one of the $B(x_i,1)$ first (as it is shown on Figure \ref{machin4}) and therefore

\begin{multline}
 \bar Z^{\go}(0,y_{r+l}) \le \sum_{i=1}^{k_r}\bE\left[e^{-\int_0^{T_{B(x_i)}}(\gl +\bar V^\go(B_t))\dd t}\right.
 \\
 \left.\ind_{\{T_{B(x_i)}\le T_{B(y_{l+r})}\}}e^{-\int_{T_{B(x_i)}}^{T_{B(y_{l+r})}}(\gl +\bar V^\go(B_t))\dd t} \ind_{\{T_{B(y_{l+r})}<\infty\}}\right].
\end{multline}

\begin{figure}[hlt]\label{machin4}
\begin{center}
\leavevmode 
\epsfxsize =14 cm
\psfragscanon
\psfrag{O}{$O$}
\psfrag{xi}{$x_i$}
\psfrag{y+r}{$y_{l+r}$}
\epsfbox{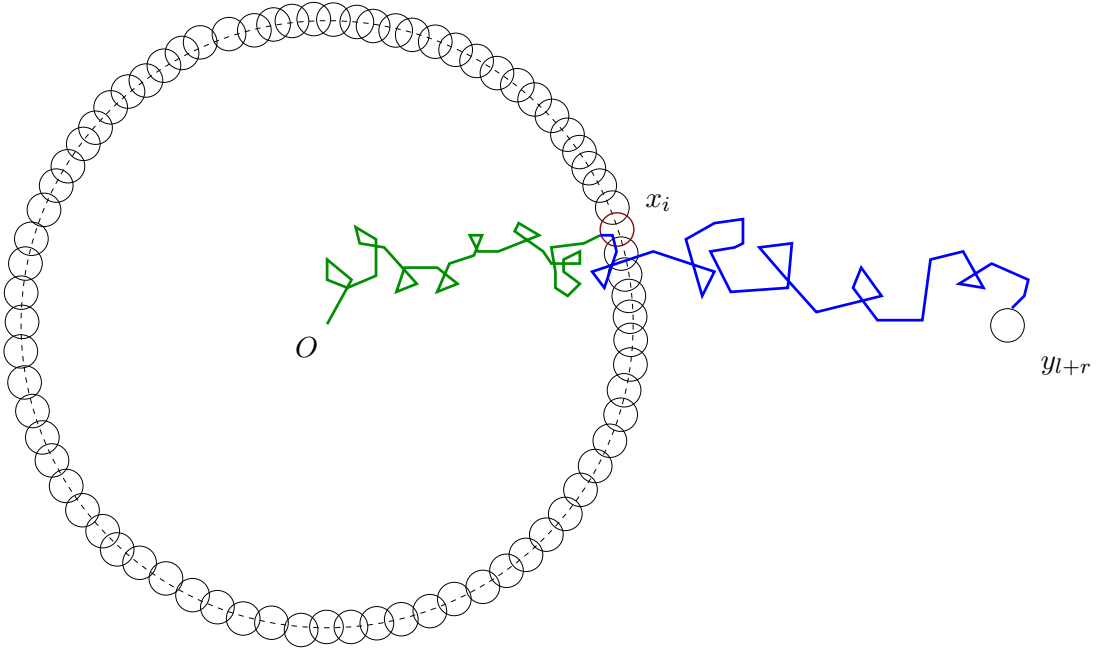}
\end{center}
\caption{In order to reach $B(y_{l+r})$, the Brownian motion starting from zero must first hit $\partial(B(0,r)$, and thus by \eqref{coc} it must it one of the $B(x_i)$.
This observation allows us to get an upper-bound on $\bar Z^\go (0,y_{r+l})$ in terms of $\bar Z^\go (x_i,y_{r+l})$ and  $\bar Z^\go (0,x_i)$.}
\end{figure}

Moreover
\begin{multline}
\bE\left[e^{-\int_0^{T_{B(x_i)}}(\gl +\bar V^\go(B_t))\dd t}
 \ind_{\{T_{B(x_i)}\le T_{B(y_{l+r})}\}}e^{-\int_{T_{B(x_i)}}^{T_{B(y_{l+r})}}(\gl +\bar V^\go (B_t))\dd t} \ind_{\{T_{B(y_{l+r})}<\infty\}}\right]\\
 \le \bE\left[e^{-\int_0^{T_{B(x_i)}}(\gl +\bar V^\go(B_t))\dd t} \ind_{\{T_{B(x_i)}<\infty\}}\right. \\ \left.
 \bE_{T_{B(x_i)}}\left[e^{-\int_{0}^{T_{B(y_{l+r})}}(\gl +\bar V^\go(B_t))\dd t} \ind_{\{T_{B(y_{l+r})}<\infty\}}
 \right]\right]\\
 \le  \bar Z^{\go}(0,x_i) 
\max_{z\in B(x_i,1)} \bar  Z^{\go}(z,y_{r+l}),
 \end{multline}
so that 
\begin{equation}\label{letetem}
 \bar Z^{\go}(0,y_{r+l}) \le 
\sum_{i=1}^{k_r} \bar Z^{\go}(0,x_i) 
\max_{z\in B(x_i,1)} \bar  Z^{\go}(z,y_{r+l}). 
\end{equation}
Now recall that 
\begin{equation}
\bar  Z^{\go}(z,y_{r+l})\le \bE\left[e^{-\gl T_{B(y_{r+l}-z)}}\ind_{\{T_{B(y_{r+l}-z)}<\infty\}}\right]
\le e^{-C_{\gl}(|y_{r+l}-z|-2)},
\end{equation}
for some constant $C_{\gl}$ (it follows from standard estimate for Brownian motion).
Then notice that for any choice of $z$ and $x_i$ one  has

 \begin{equation}
 |z-y_{r+l}|\ge |y_{r+l}|-|z|\ge |y_{r+l}|-(|z-x_i|+|x_i|)\ge l-1,
 \end{equation}
so that there exists a constant $c$ such that for all $l\ge L^{\chi+\gep}$, and $z\in B(x_i,1)$

\begin{equation}
 \bar  Z^{\go}(z,y_{r+l})\le e^{-2c l}.
\end{equation}
As a consequence 
\begin{equation}\label{matuv}
  \bar Z^{\go}(0,y_{r+l})\le (k_r e^{-2c l})\max_{i\in \{1,\dots, k_r\}} \bar Z^{\go}(0,x_i).
\end{equation}
The different $\bar Z^{\go}(0,x_i)$ are identically distributed.
 Thanks to Proposition \ref{concentrat} one can find a $\gd$ such that for all $L$ large enough
 \begin{equation}
 \bbP\left( \log \max_{i\in \{1,\dots, k_r\}} \bar Z^{\go}(0,x_i)-\alpha(r) \ge L^{\chi+(\gep/2)} \right)\le k_r e^{-L^{\gd}}.
 \end{equation}
 As we also have  that deterministically
 \begin{equation}
 \max_{i\in \{1,\dots, k_r\}} \bar Z^{\go}(0,x_i)\le 1.
 \end{equation}
This implies
\begin{equation}\label{truit}
\bbE\left[ \log  \max_{i\in \{1,\dots, k_r\}} \bar Z^{\go}(0,x_i)\right]\le -\alpha(r)+L^{\chi+\gep/2}+\alpha(r) k_r e^{-L^{\gd}}.
\end{equation}
Altogether by taking the expectation of $-\log$ of \eqref{matuv}
\begin{equation}
 \alpha(l+r)\ge \alpha(r)+2cl-\log k_r-\alpha(r) k_r e^{-L^{\gd}}-L^{\chi+\gep/2}\ge \alpha(r)+cl.
\end{equation}
where the last inequality holds when the assumption given in the Lemma for $r$ and $l$ are satisfied and $L$ is large enough.
 
 \end{proof}

\subsection{Proof of Theorem \ref{turiaf}}

The idea for the proof is the following: Set $|y|=L$,
according to Proposition \ref{micmac} it is sufficient to prove that
\begin{equation}
\bar \mu^{\go}_{0,y}((\mathcal A^{\xi}_y)^c)\Rightarrow 0, \text{ in probability when } |y|\to \infty.
\end{equation}
In order to to go out of $\mathcal C^\xi_y$ before hitting $B(y)$, 
$(B_t)_{t\ge 0}$ has to travel a longer distance that if it went in ``straight-line".
This extra distance traveled is at least of order $L^{2\xi-1}$.
Lemma \ref{witchi} allows to say that the cost of traveling is linear in the distance.
However doing this may bring some extra-energy reward by allowing to visit regions that are more favorable energetically. Proposition
\ref{concentrat} ensures that the energetic gain may not be more than $L^{\chi+\gep}$.
As with our choice of parameter
\begin{equation}
2\xi-1>\chi(\xi),
\end{equation}
the cost of extra-travel cannot be compensated by this energetic gain and this implies that 
the probability of $(\mathcal A^\xi_y)^c$ is small.
This is not too complicated to make this heuristic rigorous.

Our aim is to compare $\bar Z^\go (0,y)$ with 
\begin{equation}
Y_{0,y}=\bar Z^{\go}(0,y)\bar \mu^{\go}_{0,y}((\mathcal A^{\xi}_y)^c)= \bE\left[e^{-\int_0^{T_{B(y)}}(\gl +\bar V(B_t))\dd t}
\ind_{\{T_{\partial\mathcal C_y^\xi}<T_{B(y)}\}}\right].
\end{equation}

Let us consider a family of ball $(B(x_i,1))_{i\in \{1,\dots, m_L\}}$, $x_i\in \partial \mathcal C_y^{\xi}$ with $m_L=O( L^{(d-2)\xi+1})$ that satisfies 
\begin{equation}\label{alwaysc}
\partial \mathcal C_L^{\xi}\subset \bigcup_{i=1}^{ m_L}B(x_i,1).
\end{equation}
Trajectories in $(\mathcal A^{\xi}_y)^c$ have to hit one of the $B(x_i,1)$ before hitting $B(y)$ and therefore
 with a computation analogous to the one we made to obtain \eqref{letetem} (see figure \ref{machin2}), we get that
\begin{equation}
Y_{0,y}\le \sum_{i=1}^{m_L} \bar Z^\go(0,x_i)\max_{z\in B(x_i,1)} \bar Z^\go(z,y).
\end{equation}

\begin{figure}[hlt]\label{machin2}
\begin{center}
\leavevmode 
\epsfxsize =14 cm
\psfragscanon
\psfrag{O}{$O$}
\psfrag{xi}{$x_i$}
\psfrag{y}{$y$}
\psfrag{CLxi}{$\mathcal C_y^\xi$}
\epsfbox{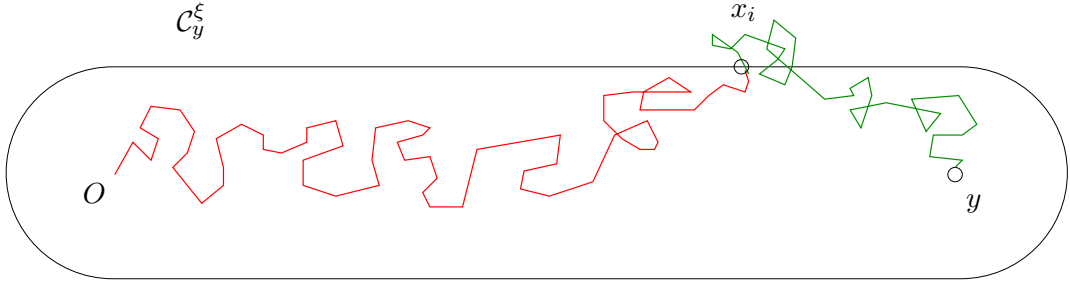}
\end{center}
\caption{If a trajectory does not belong to $\mathcal A_y^\xi$ then it has to hit $ \partial \mathcal C_y^{\xi}$ at some point 
before $T(B(y))$ and thus, by \eqref{alwaysc} it has to hit one of the $B(x_i)$.
This observation allows to get an upper bound on $Y^\go_{0,y}$.
}
\end{figure}
Note that one can find a constant $C$ such that for any $z\in B(x,1)$ (cf. (2.22) pp 225 in \cite{cf:S}),
\begin{equation}\label{tyu}
|\log \bar Z^\go(z,y)-\log \bar Z^\go(x,y)|\le C \log L.
\end{equation}
Moreover,
according to Proposition \ref{concentrat}, for any $\gep>0$ one has, for $L$ large enough,
 \begin{equation}\begin{split}\label{tyu1}
 \bbP\left( \exists i\in \{1,\dots, m_L\},\ \log \bar Z^{\go}(0,x_i)+\alpha(|x_i|) \ge L^{\chi+\gep} \right)&\le m_L e^{-L^{\gd}}\\
  \bbP\left( \exists  i\in \{1,\dots, m_L\},\ \log  \bar Z^{\go}(x_i,y)+\alpha(|y-x_i|) \ge L^{\chi+\gep} \right)&\le m_L e^{-L^{\gd}}
 \end{split}\end{equation}
so that combining \eqref{tyu} and \eqref{tyu1} one gets that with high probability

\begin{multline}\label{saei}
Y_{0,y}\le m_L e^{2L^{\chi(\xi)+\gep}+C\log L}\max_{i\in \{0,\dots, m_L\}}e^{-\alpha(|x_i|)- \alpha(|y-x_i|)}
\\ \le
 m_L e^{L^{\chi(\xi)+\gep}+C\log L}\max_{x\in \partial \mathcal C_L^\xi}e^{-\alpha(|x|)- \alpha(|y-x|)}.
\end{multline}
One also has that for any choice of $r\in [0,3L/4]$ (recall $L=|y|$), with large probability (using Proposition \ref{concentrat}
and \eqref{tyu})
\begin{equation}\label{cocose}
\bar Z(0,y)\ge \bar Z^\go (0,(r/L)y)\min_{z\in B((r/L)y,1)}\bar Z^\go(z,y)
\ge 
 e^{-2L^{\chi(\xi)+\gep}-C\log L}e^{-\alpha(r)-\alpha(L-r)}.
\end{equation}

Set  $x_0\in \argmin_ {x\in \partial \mathcal C_L^\xi}\alpha(|x|)+ \alpha(|y-x|)$. For large values of $L$, either $|x_0|\le 3L/4$ or $|y-x_0|\le 3L/4$ holds, and by symmetry one can assume that $|x_0|\le 3L/4$.
Then taking $r=|x_0|$ in \eqref{cocose} one obtains that with  probability going to one
\begin{equation}
\log \bar \mu^{\go}_{0,y}\left((\mathcal A^{\xi}_y)^c\right)=\log \frac{Y_{0,y}}{\bar Z^{\go}(0,y)} \le \alpha(L-|x_0|)-\alpha(|y-x_0|)+4L^{\chi+\gep}+C'\log L.
\end{equation}
Note that necessarily  $|y-x_0|-(L-|x_0|)\ge  L^{2\xi-1}$ for large $L$ as it is the case for any $x\in \partial C_y^\xi$ (see Figure \ref{machin3}).
\begin{figure}[hlt]
\begin{center}
\leavevmode 
\epsfxsize =14 cm
\psfragscanon
\psfrag{O}{$O$}
\psfrag{Cyxi}{$\mathcal C_y^\xi$}
\psfrag{Lxi}{$L^\xi$}
\psfrag{y}{$y$}
\psfrag{L-a}{$L-a$}
\psfrag{a}{$a$}
\epsfbox{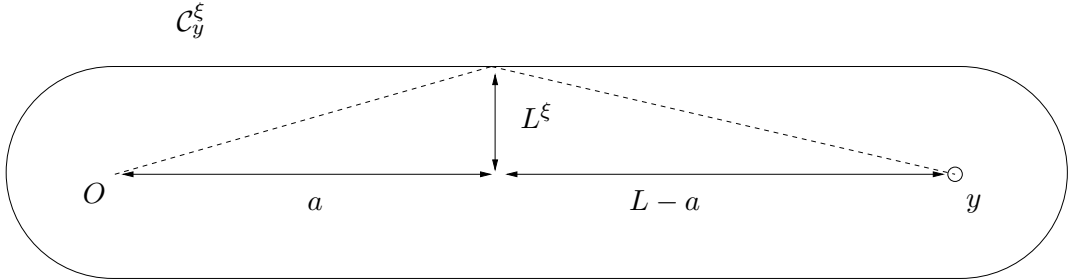}
\end{center}
\caption{\label{machin3} Suppose that $x$ is on $\partial \mathcal C_y^\xi$ for $|y|=L$ then
if $x$ is on the "cylindric" part then
$|y-x|+|x|=\sqrt{a^2+L^{2\xi}}+\sqrt{(L-a)^2+L^{2\xi}}\ge 2\sqrt{(L/2)^2+L^2\xi}= 2L^{2\xi-1}(1+o(1))$.
We let the reader check that this also holds when $x$ is on one of the ``hemispheres".
}
\end{figure}

With our choice of $\xi$
\begin{equation}
 2\xi-1>\chi(\xi)>0,
\end{equation}
so that one can use Lemma \ref{witchi} to get that 
\begin{equation}
 \alpha(L-|x|)-\alpha(|y-x|)\ge  c   L^{2\xi-1}.
\end{equation}
and hence 
\begin{equation}
\log \bar \mu^{\go}_{0,y}(\mathcal A^{\xi}_y)\le 4L^{\chi+\gep}+C'\log L- c   L^{2\xi-1}\le -\frac{c}{2}L^{2\xi-1}.
\end{equation}
\qed

\subsection{Proof of Theorem \ref{turiaf1}}\label{super2}

To treat the point to plane model needs a bit more care but the general idea is the same.
Thanks to Proposition \ref{micmac}, the result is proved if
$\bar \mu^{\go}_L((\mathcal A^{\xi}_L)^c)$ tends to zero in probability.
\medskip

Therefore our aim is to compare 

\begin{equation}
 Y^\go_{L}:=\bar Z^{\go}_L\bar \mu^{\go}_L((\mathcal A^{\xi}_L)^c\cap \mathcal S_L)
=\bE\left[e^{-\int_0^{T_{\mathcal H_L}}(\gl +\bar V(B_t))\dd t}\ind_{\{T_{\partial\mathcal C_L^\xi}<T_{\mathcal H_L}\}}
\right].
\end{equation}
with $\bar Z^{\go}_L$.
First one can remark that 
\begin{equation}
\bar Z^\go_L\ge \bar Z^\go(0,y_{L+1}),
\end{equation}
where $y_L=(L+1,0,\dots,0)$. Then one has to find a good upper bound on $Y_L$.

Consider a family of ball $(B(x_i,1))_{i\in \{0,\dots, m_L\}}$, $x_i\in \partial \mathcal C_L^{\xi}$ 
with $m_L= O(L^{1+\xi(d-2)})$ that satisfies 
\begin{equation}\label{alwaysc2}
\partial \mathcal C_L^{\xi}\subset \bigcup_{i\in \{0,\dots, m_L\}}B(x_i,1).
\end{equation}
Then for each $i$ set $r_{i,L}:= d(x_i,\mathcal H_L)$,  consider a family of balls 
$(B(y_{i,j},1))_{j\in \{0,\dots, n_{i,L}\}}$, with $y_{i,j}\in B(x_i,r_{i,l})$, $n_{i,L}=O(L^d)$
that cover entirely the bondary of $B(x_i,r_{i,l})$.
\begin{equation}\label{alwayssc3}
 \partial B(x_i,r_{i,L})\subset \bigcup_{j\in \{0,\dots, n_{i,L}\}}B(y_{i,j},1).
\end{equation}

%
%
%
Then one remarks that trajectories in $(\mathcal A^{\xi}_L)^c$ have to hit, first one of the $B(x_i,1)$
(they have to hit $\partial  \mathcal C_L^{\xi}$ first), then one of the $B(y_{i,j},1)$ (starting from $x_i$ one has 
to hit $\partial B(x_i,r_{i,L})$ before hitting $\mathcal H_L$ see Figure \ref{machin}),
so that
with a computation similar to the one made to obtain \eqref{letetem}, we obtain that
\begin{equation}
 Y_L\le \sum_{i\in \{0,\dots, m_L\}}\sum_{j\in \{0,\dots, n_{i,L}\}} \bar Z^\go(0,x_i)\max_{z\in B(x_i,1)}
\bar Z^\go(z,y_{i,j}),
\end{equation}
with the convention that $\bar Z^\go (a,b)=1$ if $|b-a|\le 1$.
\begin{figure}[hlt]\label{machin}
\begin{center}
\leavevmode 
\epsfxsize =14 cm
\psfragscanon
\psfrag{O}{$O$}
\psfrag{xi}{$x_i$}
\psfrag{L}{$L$}
\psfrag{CLxi}{$\mathcal C_L^\xi$}
\psfrag{Bxi}{$B(x_i,d(x_i,\mathcal H_L))$}
\psfrag{HL}{$\mathcal H_L$}
\epsfbox{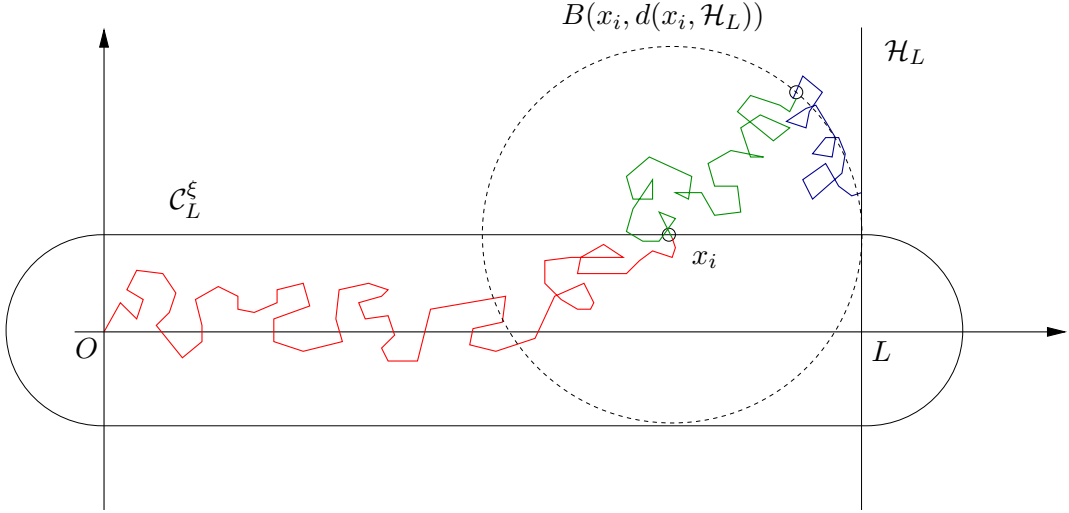}
\end{center}
\caption{If a trajectory does not belong to $\mathcal A_L^\xi$ then it has to hit $ \partial \mathcal C_y^{\xi}$ at some point 
before $T_{\mathcal H_L}$ and thus, by \eqref{alwaysc2} it must hit one of the $B(x_i)$. Then before hitting $\mathcal H_L$
is has to hit $\partial B(x_i,d(x_i,\mathcal H_L)$ (because of distance consideration) and thus, by \eqref{alwayssc3} one of the
$B(y_{i,j})$. We use this information to get an upper bound on $Y^\go_L$.}
\end{figure}
Then recall \eqref{tyu}
\begin{equation}
 \max_{z\in B(0,x_i)}\bar Z^\go(z,y_j)\le e^{c\log L} Z^\go(x_i,y_{i,j}),
\end{equation}
and that concentration inequalities from Proposition \ref{concentrat} tells us that with high-probability, all the 
$\log \bar Z^\go(0,x_i)$
and $\log \bar Z^\go(x_i,y_{i,j})$ are not further than $L^{\chi+\gep}$ away from their respective mean value, or more precisely

\begin{equation}
 \bbP\left( \exists i\in \{1,\dots, m_L\},\ \log \bar Z^{\go}(0,x_i)+\alpha(|x_i|) \ge L^{\chi+\gep} \right)\le m_L e^{-L^{\gd}},
\end{equation}
and
\begin{multline}
  \bbP\left( \exists  i\in \{1,\dots, m_L\}, \exists j\in \{0,\dots,n_{i,L}\} ,\
\log  \bar Z^{\go}(x_i,y_{i,j})+\alpha(|y_{i,j}-x_i|) \ge L^{\chi+\gep} \right)\\
\le m_L(\max_{i} n_{i,L}) e^{-L^{\gd}}.
 \end{multline}

Hence similarly to \eqref{saei} there exists a constant $C'$ such that with high probability
\begin{equation}
 \log Y^\go_L\le C' \log L+2L^{\chi+\gep}-
\min_{x\in \partial \mathcal C_L^{\xi}} \left(\alpha(|x|)+\alpha(d(x,\mathcal H_L))\right). 
\end{equation}
Consider $x_0\in \argmin_{x\in \partial \mathcal C_L^{\xi}} \alpha(|x|)+\alpha(d(x,\mathcal H_L))$.
Note that either $|x_0|$ or $d(x_0,\mathcal H_L)$ is smaller than $3L/4$. Suppose  $|x_0|\le 3/4L$
(the proof would work the same way in the other case).
For any $r\in [0,3L/4]$ one has that with high probability (cf. \eqref{cocose}),
\begin{equation}\label{cocose2}
\bar Z_L^{\go}\ge \bar Z^\go(0,y_{L+1})\ge
 e^{-2L^{\chi(\xi)+\gep}-C\log L}e^{\alpha(|x_0|)+\alpha(L+1-|x_0|)},
\end{equation}
so that 

\begin{equation}
\log \bar \mu^{\go}_{0,L}\left((\mathcal A^{\xi}_L)^c\right)=\log Y_L^\go/\bar Z_L^\go\le
\alpha(L+1-|x_0|)-\alpha(d(x_0,\mathcal H_L))+4L^{\chi+\gep}+C'\log L.
\end{equation}
From geometric consideration as $x_0\in \partial \mathcal C_L^\xi$ one has
\begin{equation}
 |x_0|+d(x_0,\mathcal H_L)\ge \sqrt{L^2+L^{2\xi}}\ge L+\frac{1}{4}L^{2\xi-1}.
\end{equation}
So that from lemma \ref{witchi} (as $2\xi-1>\chi(\xi)$) 
\begin{equation}
  \alpha(d(x_0,\mathcal H_L))-\alpha(L+1-|x_0|)\ge c\left(\frac{1}{4}L^{2\xi-1}-1\right).
\end{equation}
and hence with high probability, provided $\gep$ is small enough
\begin{equation}
\log \bar \mu^{\go}_{0,L}\left((\mathcal A^{\xi}_L)^c\right)\le -c\left(\frac{1}{4}L^{2\xi-1}-1\right)+4L^{\chi+\gep}+C'\log L\le -\frac{cL^{2\xi-1}}{8}.
\end{equation}

\qed

%
%

%
%
%

\appendix

\section{Technical estimates}

We present here the proof of two technical statement. The first one, Lemma \ref{techin}, is 
the fact that with our setup, each $x$ in $B(0,L^2)$ lies in at most $\log L$ different traps 
with high probability.

\medskip

The second statement Lemma \ref{tuture} is that under our Gibbs measure, $B_t$ does not visit too many different cubes of side-length $l$.

\begin{lemma}\label{techin}
One has  that for all $L$ large enough, 
\begin{equation}
\bbP\left[  \max_{x\in B(0,L^2)} \left(\sum_{i=1}^\infty \ind_{\{| x-\go_i | \le r_i\}}\right)\ge \log L\right]
\le \frac{1}{L}
\end{equation}

\end{lemma}

\begin{proof}
Note that it is sufficient to show that 
\begin{equation}
\bbP\left[  \max_{x\in B(0,1)} \sum_{i=1}^\infty \ind_{\{| x-\go_i | \le r_i\}}\ge \log L\right]\le \frac 1 {L^{2d+2}}
\end{equation}
Indeed, one can cover up $B(0,L^2)$ with $O(L^{2d})$ balls of radius one, and use union bound and translation invariance.
Then we remark that $\max_{x\in B(0,1)}\dots$ is less than
\begin{equation}
\left(\sum_{i=1}^\infty \ind_{\{| \go_i | \le r_i+1\}}\right)
\end{equation}
which is a Poisson variable whose mean is
\begin{equation}
\int_{1}^\infty \alpha r^{-\alpha-1} \sigma_d (r+1)^d \dd r<\infty,
\end{equation}
this is enough to conclude.
\end{proof}

For $l\ge 0$,
$x\in \bbZ^d$ define $C_x:= lx +[0,l]^d$ and $\tilde C_x:=\bigcup_{y\in C_y}B(x,l)$.
For $T\ge 0$ define
\begin{equation}
A_T:=| \{ x\in \Z^d \ | \ \tilde C_x\cap \{ B_t\ ,\ t\in [0,T]\}|,
\end{equation}
the number of $\tilde C_x$ that are visited by $(B_{t})_{t\in[0,T]})$.
Scaling properties of the Brownian motion implies that $A_T$ is typically of order $O(T/l^2)$ 
(and smaller than this when $B$ is recurent, i.e.\ for $d=1,2$).
We investigate large deviation of $A_T$ above its typical value

\begin{lemma}
There exist a constant $C$ such that if $nl^2/T\ge C$ then
\begin{equation}
\bP\left[ A_T\ge n \right]\le e^{-\frac{n^2l^2}{4CT}}.
\end{equation}

\end{lemma}
\begin{proof}

Set $\cT_0:=0$ and

\begin{equation}
\cT_{n+1}:=\inf\{t\ge \cT_{n},\ |B_t-B_{\cT_n}|\ge l\}.
\end{equation}
Note that in the interval $(\cT_n,\cT_{n+1})$ the Brownian motion cannot visit more than $5^d$ different $\tilde C_x$, and therefore

\begin{equation}
\bP\left[ A_T\ge 5^d n \right]\le \bP\left[ \cT_n \le T \right].
\end{equation}
To estimate the second term, one uses Chernov inequality and therefore, the first step is to compute the Laplace transform of $\cT_1$

\begin{multline}
\bE\left[e^{-u \cT_1}\right]=\int_0^\infty u e^{-u t}\bP\left[ \cT_1 \ge t \right]\dd t
\le 4d \int_0^\infty u e^{-u t}\int_{l/ \sqrt{t}}^{\infty}\frac{1}{\sqrt{2\pi}}e^{-\frac{x^2}{2}}\dd x \dd t
\\
\le C \int_0^{\infty} \frac{\sqrt{t}}{l}e^{- u t-\frac{l^2}{2t}}\dd t \le e^{-l\sqrt{u}},
\end{multline}
where the last inequality holds if $l^2 u$ is large enough, say larger than a constant $C$.

\begin{equation}
\bP\left[ \cT_n \le T \right]\le \inf_{u\ge 0} (\bP\left[e^{-u \cT_1+uT/n}\right])^n
\le  \inf_{u\ge C/l^2} (\bP\left[e^{-l\sqrt{u}+uT/n }\right])^n=e^{-\frac{n^2l^2}{4T}}
\end{equation}
where the last equality holds provided $nl^2/T$ is large.
\end{proof}

We use the previous estimate to get a (rather rough) bound on the tail distribution of $A_{T_{B(v)}}$ under $\bar \mu^\go_{0,v}$.
\begin{lemma}\label{tuture}
There exists a constant $C$ such that for all $L$ large enough and all $|v|\in \bbR^d$,
for all $l\ge 1$ and  for all $n\ge \frac{C|v|\log L }{l}$ 
\begin{equation}
\bar \mu^\go_{0,v}\left[ A_{T_{B(v)}}\ge n\right]\le e^{\frac{-nl}{C}}
\end{equation}
\end{lemma}

\begin{proof}
First recall that via standard tubular estimates for Brownian Motion, one can prove that almost surely, for all $v$ 
\begin{equation}
\log \bar Z^\go_{v}\ge -C  |v| \log L. 
\end{equation}

And therefore
\begin{equation}
\mu_{0,v}(T_{B(v)}\le T)\le e^{-\gl T+ |v| \log L}. 
\end{equation}
On the other hand if $nl^2/T\ge C$
\begin{equation}
\bar \mu ^\go_{0,v}\left[ A_{T}\ge n\right] \le \frac{1}{\bar Z^\go_y}\bP\left[A_T\ge n\right]
\le e^{|v|\log L-\frac{n^2l^2}{4T}}.
\end{equation}
Altogether one has that
\begin{equation}
\bar \mu ^\go_{0,v}\left[ A_{T_{B(v)}}\ge n\right]\le \bar \mu ^\go_{0,v}\left[ A_{T}\ge n\right]+\bar \mu ^\go_{0,v}\left[ T_{B(v)}\ge T\right]
\le e^{C|v| \log L}\left( e^{-\gl T}+ e^{-\frac{n^2l^2}{4TC}}\right).
\end{equation}
were the last inequality is valid when $T\le nl^2/C$. Taking $T=nl/C$ one gets that
\begin{equation}
\bar \mu ^\go_{0,v}\left[ A_{T_{B(v)}}\ge n\right] \le e^{Cv|\log L|}(e^{\frac{\gl n l}{C}}+e^{-\frac{nl}{4}})\le e^{\frac{\gl n l}{2C}}
\end{equation}
where the last inequality holds if $C$ is large enough and $n\ge \frac{2C^2|v|\log L }{l}$.

\end{proof}

\end{document}